\documentclass[12pt]{amsart}
\usepackage[utf8]{inputenc}
\usepackage[T1]{fontenc}
\usepackage{amsmath,amssymb,amsthm}
\usepackage{mathrsfs}
\usepackage{xcolor}
\usepackage{enumitem}

\definecolor{sectioncolor}{rgb}{0,0,0}
\definecolor{definitioncolor}{rgb}{0,0,0}

\newtheorem{theorem}{Theorem}[section]
\newtheorem{lemma}[theorem]{Lemma}

\theoremstyle{definition}

\theoremstyle{remark}
\newtheorem*{acknowledgment}{Acknowledgments}

\makeatletter
\renewcommand{\section}{\@startsection{section}{1}{\z@}%
  {-3.5ex \@plus -1ex \@minus -.2ex}%
  {2.3ex \@plus.2ex}%
  {\normalfont\large\bfseries\color{sectioncolor}}}
\makeatother


\begin{document}

\title[On the asymptotic behavior of the integral $T(x)$]{On the asymptotic behavior of the integral $\displaystyle T(x)=\int_{0}^{+\infty}\left|\left\{\frac{x}{t}\right\}-\left\{\frac{x}{t+1}\right\}\right|\,dt$}

\author{Mihoub BOUDERBALA}
\address{(1) Khemis Miliana University (UDBKM), FIMA Laboratory, Faculty of Matter Sciences and Computer Sciences, Rue Thniet El Had, Khemis Miliana, 44225, Ain Defla Province, Algeria}
\email{mihoub75bouder@gmail.com}

\author{Meselem KARRAS}
\address{(2) Tissemsilt University, Faculty of Science and Technology, FIMA Laboratory, Khemis Miliana, Algeria}
\email{m.karras@univ-tissemsilt.dz}

\maketitle

\begin{abstract}
In this paper, we estimate the integral $T(x)$ mentioned in the title, where $\{t\}$ denotes the fractional part of the real number $t$, and $x$ is any positive real number.
\end{abstract}

\bigskip


\bigskip

\section{Introduction and Results}

The study of sums involving fractional parts aroused great interest among many researchers, particularly in the late 19th century. One of the key results of this period is the asymptotic formula for the sum of fractional parts, derived from Dirichlet's seminal work \cite{Apostol1976} on the divisor function $\tau(n)$,
\[
\sum_{n\le x}\left\{\frac{x}{n}\right\}=(1-\gamma)x+O\!\left(x^{1/2}\right),
\tag{1}
\]
where $\gamma$ is the Euler--Mascheroni constant. This result paved the way for new explorations of these sums, and insights were gained into approximate formulas and error bound estimates for sums involving fractional parts. Recently, the refinement of error terms for (1) has been considerably advanced by Bourgain \cite{BourgainWatt2017}; he showed
\[
\sum_{n\le x}\left\{\frac{x}{n}\right\}=(1-\gamma)x+O\!\left(x^{517/1648+\varepsilon}\right),
\]
where $\varepsilon>0$ is arbitrarily small.

The study of the variation between two successive terms, $\bigl\{x/n\bigr\}$ and $\bigl\{x/(n+1)\bigr\}$, was initially carried out by Aurel Wintner \cite{Wintner1946}. Later, Balazard in \cite{Balazard2017}, established the following result
\[
\sum_{n\ge1}\left|\left\{\frac{x}{n}\right\}-\left\{\frac{x}{n+1}\right\}\right|=\frac{2}{\pi}\zeta(3/2)\sqrt{x}+O\!\left(x^{2/5}\right),\qquad(x>0).
\tag{2}
\]
Balazard et al. \cite{BalazardBenferhatBouderbala2021} extended the result (2) by studying a more general case. They established the following asymptotic formula
\[
\sum_{n\ge0}\left|\left\{\frac{x}{n+a}\right\}-\left\{\frac{x}{n+b}\right\}\right|=\frac{2}{\pi}\zeta(3/2)\sqrt{cx}+O\!\left(c^{2/9}x^{4/9}\right),
\tag{3}
\]
where $b>a>0$, $c=b-a$, and the result holds uniformly for $x\ge40c^{-5}(1+b)^{27/2}$.

The connection between sums and integrals has been a fundamental topic, with the Riemann--Stieltjes integral serving as a particularly powerful tool in this context. One of its key advantages lies in its simplicity, as demonstrated by the following result:
\[
\sum_{\alpha<n\le\beta}f(n)=\int_{\alpha}^{\beta}f(t)\,dt+O\bigl(|f(\alpha)|+|f(\beta)|\bigr),
\]
where $\alpha,\beta\in\mathbb{Z}$ and $f$ is a real-valued monotonic function on $[\alpha,\beta]$. If the function $f$ is not monotonic, the problem poses some difficulties. In this case, Wu and Shi \cite{WuShi2020} focused on refining the error term in the difference:
\[
\sum_{n\le x}f(n)\left\{\frac{x}{n}\right\}^{k}-\int_{1}^{x}f(t)\left\{\frac{x}{t}\right\}^{k}dt.
\]
Building on this, Mercier and Nowak \cite{MercierNowak1985} derived the following bound:
\[
\sum_{n\le x}f(n)\left\{\frac{x}{n}\right\}^{k}=\int_{1}^{x}f(t)\left\{\frac{x}{t}\right\}^{k}dt+O\!\left(f(x)x^{131/416}(\log x)^{26947/8320}\right),
\]
where $f(t)$ is a real-valued, positive, nondecreasing function defined for $t\ge1$, and $k$ is an arbitrary positive integer.

Motivated by this result, we explore the asymptotic behavior of the following integral.
\[
T(x)=\int_{0}^{+\infty}\left|\left\{\frac{x}{t}\right\}-\left\{\frac{x}{t+1}\right\}\right|\,dt.
\]
We now state our principal result.

\begin{theorem}
For all $x>0$, we have
\[
T(x)=\frac{2}{\pi}\zeta(3/2)\sqrt{x}+O\!\left(x^{13/30}\right),
\tag{4}
\]
where $\zeta$ denotes the Riemann zeta function.
\end{theorem}

Certain formulas from Balazard's original article \cite{Balazard2017} are repeated here to ensure the clarity and coherence of our presentation.

\section{Proof of Theorem 1}

Let $x>1$ be a real number. We define the function
\[
T(x)=\int_{0}^{+\infty}\left|\left\{\frac{x}{t}\right\}-\left\{\frac{x}{t+1}\right\}\right|\,dt.
\]
This function can also be rewritten as
\[
T(x)=\int_{0}^{+\infty}\left|\left\lfloor\frac{x}{t}\right\rfloor-\left\lfloor\frac{x}{t+1}\right\rfloor-\frac{x}{t(t+1)}\right|\,dt.
\]
By setting $k=\bigl\lfloor x/t\bigr\rfloor$ and $h=\bigl\lfloor x/(t+1)\bigr\rfloor$, it follows that $t$ belongs to the interval
\[
t\in I(h,k;x)=\left(\frac{x}{k+1},\frac{x}{k}\right]\cap\left(\frac{x}{h+1}-1,\frac{x}{h}-1\right],
\tag{5}
\]
where $k\ge h\ge0$. As a result, the function $T(x)$ can be expressed as the sum
\[
T(x)=\sum_{k\ge h\ge0}\left(\int_{t\in I(h,k;x)}\left|k-h-\frac{x}{t(t+1)}\right|\,dt\right).
\]
Now, for a positive integer $d$, we define the set
\[
E_{d}=\bigl\{(h,k)\in\mathbb{N}^{2}:0\le h\le k,\ k-h=d\bigr\}
=\bigl\{(k-d,k):k\in\mathbb{N},\ d\le k\bigr\}.
\tag{6}
\]
We also introduce the auxiliary function
\[
T_{d}(x)=\sum_{(h,k)\in E_{d}}\left(\int_{t\in I(h,k;x)}\left|k-h-\frac{x}{t(t+1)}\right|\,dt\right)
=\sum_{d\le k}\left(\int_{t\in I(k-d,k;x)}\left|d-\frac{x}{t(t+1)}\right|\,dt\right).
\tag{7}
\]
Thus, we obtain the final expression
\[
T(x)=\sum_{d\ge0}T_{d}(x).
\]
To derive an asymptotic formula for the sum $T(x)$ defined by the previous expression, we introduce a large positive real number $D$, and represent the sum $T(x)$ as follows:
\[
T(x)=T_{0}(x)+\sum_{1\le d\le D}T_{d}(x)+\sum_{d>D}T_{d}(x).
\]
In our next plan, we will study separately the three partial sums $T_{0}(x)$, $\sum_{1\le d\le D}T_{d}(x)$ and $\sum_{d>D}T_{d}(x)$ by finding their asymptotic formulas, ultimately obtaining an asymptotic formula for sum $T(x)$.

\subsection{Exploring the behavior of the sum $T_{0}(x)$}

\begin{lemma}
For all $x>0$, we have
\[
T_{0}(x)=\frac{2}{3}\sqrt{x}+O(1).
\tag{8}
\]
\end{lemma}

\begin{proof}
We note that in the case where $d=0$, we may have $h=k$, which leads to
\[
t\in I(k,k;x)=\left(\frac{x}{k+1},\frac{x}{k}-1\right]
\quad\text{(see (5)).}
\]
We note that the interval $I(k,k;x)$ is non empty if and only if $k(k+1)\le x$. Let us denote by $K=K(x)$ the greatest integer $k$ such that $k(k+1)\le x$, i.e.,
\[
K=\bigl\lfloor\sqrt{x+1/4}-1/2\bigr\rfloor.
\]
Notice that $K=\sqrt{x}+O(1)$ and $K\ll\sqrt{x}$.

So, according to (7), we obtain the following expression for $T_{0}(x)$, i.e.,
\[
T_{0}(x)=\sum_{0\le k\le K}\int_{x/(k+1)}^{x/k-1}\frac{x}{t(t+1)}\,dt
=x\sum_{0\le k\le K}\int_{x/(k+1)}^{x/k-1}\left(\frac{1}{t^{2}}-\frac{1}{t^{3}}+O\!\left(\frac{1}{t^{4}}\right)\right)dt.
\]
Simplifying this expression, we get
\[
T_{0}(x)=x\sum_{0\le k\le K}\left(\frac{1}{x}-\frac{k^{2}}{x^{2}}+\frac{2k+1}{2x^{2}}+O\!\left(\frac{(k+1)^{3}}{x^{3}}\right)\right).
\tag{9}
\]
To explore the behavior of $T_{0}(x)$, we begin by noting that
\[
x\sum_{0\le k\le K}\frac{1}{x}=K+1=\sqrt{x}+O(1).
\tag{10}
\]
Moreover, we have on the other hand
\[
x\sum_{0\le k\le K}\left(-\frac{k^{2}}{x^{2}}+\frac{2k+1}{2x^{2}}+O\!\left(\frac{(k+1)^{3}}{x^{3}}\right)\right)
=-\frac{K^{3}+O(K^{2})}{3x}+O\!\left(\frac{K^{2}}{x}\right)+O\!\left(\frac{K^{4}}{x^{2}}\right)\]
\[=-\frac{1}{3}\sqrt{x}+O(1).
\tag{11}
\]
The conjunction of (9), (10) and (11) gives (8).
\end{proof}

\subsection{Exploring the behavior of the sums $T_{d}(x)$ for $d\ge1$}

The key to this step of the proof is first to specify the intervals $I(h,k;x)$ mentioned in (5) without using intersections. Notably, if $k>h\ge0$ and $x>0$, we have
\[
\frac{x}{k}\le\frac{x}{h}-1\iff kh\le(k-h)x.
\]
In particular, this leads to the implications
\[
\frac{x}{k+1}\le\frac{x}{h+1}-1\implies\frac{x}{k}\le\frac{x}{h}-1
\quad\text{and}\quad
\frac{x}{k}>\frac{x}{h}-1\implies\frac{x}{k+1}>\frac{x}{h+1}-1.
\]
These observations naturally lead us to consider only three distinct intervals, systematically linked to three distinct parts of the set $E_{d}$ defined in (6). So, for each particular interval, we associate its corresponding subset, denoted as $E_{d,i}$ $(i=1,2,3)$, knowing that $h=k-d$.

In the case where $\frac{x}{k+1}\le\frac{x}{h+1}-1$, we obtain
\[
I(k-d,k;x)=\left(\frac{x}{k-d+1}-1,\frac{x}{k}\right],
\tag{12}
\]
and
\[
E_{d,1}=\bigl\{(k-d,k):d\le k,\ (k-d+1)(k+1)\le xd\bigr\}.
\]
In the case where $\frac{x}{k}>\frac{x}{h}-1$, we obtain
\[
I(k-d,k;x)=\left(\frac{x}{k+1},\frac{x}{k-d}-1\right],
\tag{13}
\]
and
\[
E_{d,2}=\bigl\{(k-d,k):d\le k,\ (k-d)k>xd\bigr\}.
\]
Moreover, when
\[
I(k-d,k;x)=\left(\frac{x}{k+1},\frac{x}{k}\right],
\tag{14}
\]
we obtain
\[
E_{d,3}=\bigl\{(k-d,k):d\le k,\ (k-d)k\le xd<(k-d+1)(k+1)\bigr\}.
\]
Now, for $x>0$ and $k>d\ge0$, we define $K_{d}=K_{d}(x)$ as the largest integer $k$ satisfying $(k-d)k\le dx$. So, we have $K_{0}=0$, and in general
\[
K_{d}=\left\lfloor\frac{d+\sqrt{d^{2}+4dx}}{2}\right\rfloor.
\]
The use of the function $K_{d}$ allows us to give fluidity to the quadratic conditions associated with the number $k$, which contribute to the definition of the subsets $E_{d,i}$ $(i=1,2,3)$, i.e.,
\begin{align}
E_{d,1}&=\bigl\{(k-d,k):d\le k,\ k\le K_{d}-1\bigr\},
\tag{15}\\
E_{d,2}&=\bigl\{(k-d,k):d\le k,\ k>K_{d}\bigr\},
\tag{16}\\
E_{d,3}&=\bigl\{(k-d,k):d\le k,\ k=K_{d}\bigr\}.
\tag{17}
\end{align}
Summarising the above results allows us to write $T_{d}(x)$ defined in (7), as follows
\[
T_{d}(x)=T_{d,1}(x)+T_{d,2}(x)+T_{d,3}(x),
\]
where
\[
T_{d,i}(x)=\sum_{(k-d,k)\in E_{d,i}}\left(\int_{t\in I(k-d,k;x)}\left|d-\frac{x}{t(t+1)}\right|\,dt\right)
\qquad(i=1,2,3).
\tag{18}
\]
In the process of finding asymptotic formulas for auxiliary functions $T_{d,i}(x)$ $(i=1,2,3)$, we must first consider how the absolute value is omitted in the expression of each function. For this, we define the function $N_{d}=N_{d}(x)$ as the largest real number verifying $t(t+1)\le x/d$, such that $d>0$. This leads to the following expression
\[
N_{d}=\frac{-1+\sqrt{1+4x/d}}{2}.
\]
The number $N_{d}$ obtained can be expressed in relation to the number $K_{d}$ according to \cite[Proposition~4, p.~10]{Balazard2017}, as follows
\[
\frac{x}{K_{d}+1}\le N_{d}\le\frac{x}{K_{d}},
\tag{19}
\]
in order to pave the way for the upcoming work.

We will now find the final expressions of the auxiliary functions $T_{d,i}(x)$ $(i=1,2,3)$, which will then allow us to find the expression of the function $T_{d}(x)$, and finally work on determining its asymptotic formula. Throughout paragraphs 2.2.1 to 2.2.4, let us note that $d>0$.

\subsubsection{Obtaining the expression of the sum $T_{d,1}(x)$}
Firstly, it is important to note that the definition of the sum $T_{d,1}(x)$, given by formula (18) for $i=1$, is linked to the set $E_{d,1}$ defined in (15), as well as to the interval given in (12). The latter is non-empty if and only if $\frac{x}{k-d+1}-1<\frac{x}{k}$, or, in other words, if and only if $k>K_{d}-1$. Moreover, we know that if $k\le K_{d}-1$, it follows from (19) and the definition of $K_{d}$, that the variable $t$ satisfies the inequality
\[
t\ge\frac{x}{K_{d}-d}-1\ge\frac{x}{K_{d}}\ge N_{d}.
\]
Therefore,
\[
T_{d,1}(x)=\sum_{K_{d}-1<k\le K_{d}-1}\int_{x/k}^{x/(k-d+1)-1}\left(d-\frac{x}{t(t+1)}\right)dt\]
\[=\sum_{K_{d}-1<k\le K_{d}-1}\left(d\left(\frac{k}{x}-\frac{k-d+1}{x}+1\right)+x\log\frac{1-(k-d+1)/x}{1+k/x}\right).
\]
We continue, this time focusing on the telescopic identity
\[
\sum_{a<k\le b}(u_{k}-u_{k-r})=\sum_{b-r<k\le b}u_{k}-\sum_{a-r<k\le a}u_{k}.
\]
This identity will also be applied in the following paragraphs.

So, we find that
\begin{equation}
\begin{aligned}
T_{d,1}(x)=&\;d\sum_{K_{d}-d<k\le K_{d}-1}\frac{x}{k}
-d\sum_{K_{d}-1-d+1<k\le K_{d}-1}\frac{x}{k}
+d(K_{d}-K_{d-1}-1)\\
&+x\sum_{K_{d}-1<k\le K_{d}-1}\log\frac{1-(k-d+1)/x}{1+k/x}.
\end{aligned}
\tag{20}
\end{equation}

\subsubsection{Obtaining the expression of the sum $T_{d,2}(x)$}
We recall that the sum $T_{d,2}(x)$, defined by formula (18) for $i=2$, is associated with the set $E_{d,2}$ introduced in (16), as well as with the interval specified in (13). The latter is non-empty if and only if $\frac{x}{k+1}\le\frac{x}{k-d}-1$, which is equivalent to the condition $k\le K_{d+1}-1$. Furthermore, when $k>K_{d}$, it follows from (19), and the definition of $K_{d}$, that the variable $t$ satisfies the inequality
\[
t\le\frac{x}{K_{d}+1-d}-1\le\frac{x}{K_{d}+1}\le N_{d}.
\]
Therefore,
\[
T_{d,2}(x)=\sum_{K_{d}<k\le K_{d+1}-1}\int_{x/(k+1)}^{x/(k-d)-1}\left(\frac{x}{t(t+1)}-d\right)dt\]
\[=x\sum_{K_{d}<k\le K_{d+1}-1}\log\frac{1-(k-d)/x}{1+(k+1)/x}
-\sum_{K_{d}<k\le K_{d+1}-1}d\left(\frac{x}{k-d}-\frac{x}{k+1}-1\right).
\]
More explicitly, we obtain
\begin{equation}
\begin{aligned}
T_{d,2}(x)=&\;x\sum_{K_{d}<k\le K_{d+1}-1}\log\frac{1-(k-d)/x}{1+(k+1)/x}
+d\sum_{K_{d+1}-d-1<k\le K_{d+1}}\frac{x}{k}\\
&-d\sum_{K_{d}-d<k\le K_{d}+1}\frac{x}{k}
+d(K_{d+1}-K_{d}-1).
\end{aligned}
\tag{21}
\end{equation}

\subsubsection{Obtaining the expression of the sum $T_{d,3}(x)$}
Currently, the function $T_{d,3}(x)$ is defined through relation (18) for $i=3$, and is associated with the set defined in (17). Through this, and from (14), the interval of real variable $t$ is given by $\bigl(x/(K_{d}+1),x/K_{d}\bigr]$. Therefore,
\[
T_{d,3}(x)=\int_{x/K_{d}}^{x/(K_{d}+1)}\left|d-\frac{x}{t(t+1)}\right|dt
=\int_{N_{d}}^{x/(K_{d}+1)}\left(\frac{x}{t(t+1)}-d\right)dt
+\int_{x/K_{d}}^{N_{d}}\left(d-\frac{x}{t(t+1)}\right)dt.
\]
More precisely,
\begin{equation}
T_{d,3}(x)=x\log\frac{1+(K_{d}+1)/x}{1+K_{d}/x}
-2x\log\left(1+\frac{1}{N_{d}}\right)
-d\left(2N_{d}-\frac{x}{K_{d}+1}-\frac{x}{K_{d}}\right).
\tag{22}
\end{equation}

\subsubsection{Final expression and behavior of $T_{d}(x)$}
Adding (20), (21) and (22) and reducing, we get
\begin{equation}
\begin{aligned}
T_{d}(x)=&\;d\sum_{K_{d+1}-d-1<k\le K_{d+1}}\frac{x}{k}
-d\sum_{K_{d-1}-d+1<k\le K_{d-1}}\frac{x}{k}\\
&+x\sum_{K_{d-1}<k\le K_{d+1}}\log\frac{1+k/x}{1-(k-d)/x}\\
&-x\log\left(1-\frac{K_{d+1}-d}{x}\right)\\
&-x\log\left(1-\frac{K_{d-1}-d+1}{x}\right)\\
&+d(K_{d+1}-K_{d-1}-2N_{d}-2)
-2x\log\left(1+\frac{1}{N_{d}}\right).
\end{aligned}
\tag{23}
\end{equation}
Finding the behavior of the sum $T_{d}(x)$ relies on the following lemmas.

\begin{lemma}
For any integer $d>0$, and $K_{d}-d<k\le K_{d}$, we have
\[
\sum_{K_{d}-d<k\le K_{d}}\frac{x}{k}=\sqrt{dx}+O(d).
\]
\end{lemma}

\begin{proof}
Note that if $K_{d}-d<k\le K_{d}$, we have $k=\sqrt{dx}+O(d)$ (see \cite[p.~8]{Balazard2017}). Therefore,
\[
\sum_{K_{d}-d<k\le K_{d}}\frac{x}{k}
=\sum_{K_{d}-d<k\le K_{d}}\frac{x}{\sqrt{dx}+O(d)}
=\sum_{K_{d}-d<k\le K_{d}}\left(\sqrt{\frac{x}{d}}+O(1)\right)
=\sqrt{dx}+O(d).\qedhere
\]
\end{proof}

\begin{lemma}
For any integer $d>0$, we have
\begin{equation}
x\sum_{K_{d-1}<k\le K_{d+1}}\log\frac{1+k/x}{1-(k-d)/x}
=\frac{(2d-1)\sqrt{d+1}-(2d+1)\sqrt{d-1}}{3}\sqrt{x}+O\!\left(d^{3/2}x^{-1/2}\right)+O(d).
\tag{24}
\end{equation}
\end{lemma}

\begin{proof}
On the one hand, we have
\[
\sum_{K_{d-1}<k\le K_{d+1}}\log\left(1+\frac{k}{x}\right)
=\sum_{K_{d-1}<k\le K_{d+1}}\left(\frac{k}{x}-\frac{k^{2}}{2x^{2}}+O\!\left(\frac{k^{3}}{x^{3}}\right)\right),
\]
and on the other hand, we have
\[
\begin{aligned}
\sum_{K_{d-1}<k\le K_{d+1}}\log\left(1-\frac{k-d}{x}\right)
&=\sum_{K_{d-1}<k\le K_{d+1}}\left(-\frac{k-d}{x}-\frac{(k-d)^{2}}{2x^{2}}+O\!\left(\frac{k^{3}}{x^{3}}\right)\right)\\
&=\sum_{K_{d-1}<k\le K_{d+1}}\left(-\frac{k}{x}+\frac{d}{x}-\frac{k^{2}}{2x^{2}}+\frac{d^{2}-2kd}{2x^{2}}+O\!\left(\frac{k^{3}}{x^{3}}\right)\right).
\end{aligned}
\]
We now use the fact that
\[
K_{d+1}-K_{d-1}=\bigl(\sqrt{d+1}-\sqrt{d-1}\bigr)\sqrt{x}+O(1)
\quad\text{(see \cite[p.~8]{Balazard2017})},
\tag{25}
\]
which leads to
\[
K_{d+1}-K_{d-1}\ll\sqrt{x/d},
\qquad
\sum_{K_{d-1}<k\le K_{d+1}}k^{3}\ll\sqrt{x/d}\cdot(dx)^{3/2}=dx^{2}.
\]
Hence, we conclude that
\[
x\sum_{K_{d-1}<k\le K_{d+1}}\log\frac{1+k/x}{1-(k-d)/x}
=\sum_{K_{d-1}<k\le K_{d+1}}\left(d+\frac{k^{2}}{x}+\frac{d^{2}-2kd}{2x}+O\!\left(\frac{k^{3}}{x^{2}}\right)\right).
\]
Then, according to propositions~3 of \cite[pp.~8--9]{Balazard2017}, we obtain (24).
\end{proof}

\begin{lemma}
For any integer $d>0$ and $x>0$, we have
\[
T_{d}(x)=f(d)\sqrt{x}+O\!\left(d^{2}\right),
\]
where
\[
f(d)=\frac{(8d+2)\sqrt{d+1}-(8d-2)\sqrt{d-1}}{3}-\frac{4}{\sqrt{d}}.
\]
\end{lemma}

\begin{proof}
We apply Lemma~2, by first changing $d$ to $d+1$, then $d$ to $d-1$, and obtain
\begin{equation}
d\sum_{K_{d+1}-d-1<k\le K_{d+1}}\frac{x}{k}
-d\sum_{K_{d-1}-d+1<k\le K_{d-1}}\frac{x}{k}
=2d\sqrt{dx}+O\!\left(d^{2}\right).
\tag{26}
\end{equation}
For the penultimate expression of (23), we have
\begin{equation}
-x\log\left(1-\frac{K_{d+1}-d}{x}\right)
-x\log\left(1-\frac{K_{d-1}-d+1}{x}\right)
=\bigl(\sqrt{d+1}+\sqrt{d-1}\bigr)\sqrt{x}+O(d).
\tag{27}
\end{equation}
For the last expression of (23), using (25), and the fact that
\[
\frac{K_{d}}{x}\le\frac{1}{N_{d}}\le\frac{K_{d}+1}{x}
\quad\text{and}\quad
N_{d}\le\sqrt{x/d},
\]
we find
\begin{equation}
\begin{aligned}
&d(K_{d+1}-K_{d-1}-2N_{d}-2)-2x\log\left(1+\frac{1}{N_{d}}\right)\\
&\qquad=d\bigl(\sqrt{d+1}-\sqrt{d-1}-2/\sqrt{d}\bigr)\sqrt{x}+O(d)
-2x\left(\sqrt{\frac{d}{x}}+O\!\left(\frac{d}{x}\right)\right)\\
&\qquad=d\bigl(\sqrt{d+1}-\sqrt{d-1}\bigr)\sqrt{x}-4\sqrt{d}\,\sqrt{x}+O(d).
\end{aligned}
\tag{28}
\end{equation}
By inserting (24), (26), (27) and (28) in (23), we obtain
\[
T_{d}(x)=f(d)\sqrt{x}+O\!\left(d^{2}\right),
\]
where
\[
f(d)=\frac{(8d+2)\sqrt{d+1}-(8d-2)\sqrt{d-1}}{3}-\frac{4}{\sqrt{d}}.\qedhere
\]
\end{proof}

\subsection{Exploring the behavior of the sum $\displaystyle\sum_{d>D}T_{d}(x)$}

\begin{lemma}
Let $D$ be a large positive real number. We have
\[
\sum_{d>D}T_{d}(x)\le\sqrt{x/(D-1)}.
\]
\end{lemma}

\begin{proof}
Let $D$ be a real number greater than $1$. If
\[
d=\left\lfloor\frac{x}{t}\right\rfloor-\left\lfloor\frac{x}{t+1}\right\rfloor>D,
\]
then it follows that
\[
\frac{x}{t^{2}}>\frac{x}{t(t+1)}=d+\left\{\frac{x}{t}\right\}-\left\{\frac{x}{t+1}\right\}>D-1.
\]
This implies the inequality $t<\sqrt{x/(D-1)}$.

As a result, the corresponding sum satisfies the bound
\[
\sum_{d>D}T_{d}(x)\le\int_{0}^{\sqrt{x/(D-1)}}dt=\sqrt{x/(D-1)}.\qedhere
\]
\end{proof}

\subsection{Proof of Theorem 1}

For $x>1$ and $D\ge2$, we have
\[
T(x)=\sum_{i=0}^{2}T_{i}(x)=T_{0}(x)+\sum_{1\le d\le D}T_{d}(x)+\sum_{d>D}T_{d}(x).
\]
Building on the results established in Lemmas~1,~4, and~5, it follows that
\[
T(x)=\frac{2}{3}\sqrt{x}+O(1)+\sum_{1\le d\le D}\bigl(f(d)\sqrt{x}+O(d^{2})\bigr)+O\!\left(\sqrt{x/D}\right).
\]
We observe that
\[
\sum_{1\le d\le D}f(d)=\sum_{d=1}^{\infty}f(d)-\sum_{d>D}f(d),
\]
where the series $\sum_{d=1}^{\infty}f(d)$ is convergent and satisfies
\[
\sum_{d=1}^{\infty}f(d)=-\frac{2}{3}+\frac{2}{\pi}\zeta(3/2).
\]
Thus, we obtain
\[
T(x)=\left(\frac{2}{3}+\sum_{d=1}^{\infty}f(d)\right)\sqrt{x}+O(D^{3})+O\!\left(\sqrt{x/D}\right),
\]
which simplifies to
\[
T(x)=\frac{2}{\pi}\zeta(3/2)\sqrt{x}+O(D^{3})+O\!\left(\sqrt{x/D}\right).
\]
Finally, the result is obtained by setting $D=x^{2/15}$.

\begin{acknowledgment}
The first author wishes to express sincere gratitude to Professor Michel Balazard for his valuable suggestions regarding the study of the integral introduced in the title.
\end{acknowledgment}

\end{document}